\documentclass[12pt]{amsart}
\usepackage[all]{xypic}
\usepackage{tikz}
\usetikzlibrary{arrows}
\usetikzlibrary{decorations.markings}

\usepackage{graphicx}
\usepackage{bm}
\usepackage{epsf}
\usepackage{verbatim} 
\usepackage{amsmath}
\usepackage{amsfonts}
\usepackage{amssymb}
\usepackage{mathrsfs}
\usepackage{amsthm}
\usepackage{newlfont}
 \newtheorem{thm}{Theorem}

 \newtheorem{prop}[thm]{Proposition}
 
 \theoremstyle{definition}
 \newtheorem{defn}[thm]{Definition}
 
 \theoremstyle{definition}

  \theoremstyle{definition}

 \theoremstyle{remark}
 \newtheorem{rem}[thm]{Remark}
 
 \theoremstyle{remark}

 \theoremstyle{definition}
 \newtheorem{example}[thm]{Example}

\numberwithin{thm}{section}
\numberwithin{equation}{section}

 \newcommand{\Mat}{\mathrm{Mat}}

 \newcommand{\Hom}{\mathrm{Hom}}

 \newcommand{\Frob}{\mathrm{Frob}}

 \newcommand{\Aut}{\mathrm{Aut}}
 
 \newcommand{\End}{\mathrm{End}}

 \newcommand{\ord}{\mathrm{ord}}

 \newcommand{\Gal}{\mathrm{Gal}}
 \newcommand{\GL}{\mathrm{GL}}

 \newcommand{\sep}{\mathrm{sep}}

 \newcommand{\rank}{\mathrm{rank}}

 \newcommand{\Tr}{\mathrm{Tr}}

 \renewcommand{\mod}{\mathrm{mod}}

 \newcommand{\disc}{\mathrm{disc}}
 \newcommand{\fl}{\mathfrak l}

 \newcommand{\fp}{\mathfrak p}
 \newcommand{\fq}{\mathfrak q}

 \newcommand{\fM}{\mathfrak M}

 \newcommand{\cO}{\mathcal{O}}

 \newcommand{\cE}{\mathcal{E}}

 \newcommand{\cF}{\mathcal{F}}

 \newcommand{\F}{\mathbb{F}}

 \newcommand{\Nr}{\mathrm{Nr}}
 \renewcommand{\disc}{\mathrm{disc}}

 \newcommand{\To}{\longrightarrow}

  \newcommand{\Mod}[1]{\ (\mathrm{mod}\ #1)}

\begin{document}
	
	\title[Endomorphism rings and Frobenius matrices of Drinfeld modules]{Computing endomorphism rings and Frobenius matrices of Drinfeld modules}
	
	\author{Sumita Garai}
	\address{
		Department of Mathematics, 
		Pennsylvania State University,
		University Park, PA 16802, USA
	} 
	\email{sxg386@psu.edu}

	\author{Mihran Papikian}
	\address{
		Department of Mathematics, 
		Pennsylvania State University,
		University Park, PA 16802, USA
	} 
	\email{papikian@psu.edu}

	%\thanks{The author's research was partially supported by grants from the Simons Foundation (245676) and the National Security Agency 
	%(H98230-15-1-0008).} 
	
	\subjclass[2010]{11G09, 11R58}
	
	\keywords{Drinfeld modules, endomorphism rings, Gorenstein rings}
	
	\begin{abstract}  
		Let $\mathbb{F}_q[T]$ be the polynomial ring over a finite field $\mathbb{F}_q$. 
		We study the endomorphism rings of Drinfeld $\mathbb{F}_q[T]$-modules of arbitrary rank over finite fields. We compare 
		the endomorphism rings to their subrings generated by the Frobenius endomorphism and deduce from this 
		a refinement of a reciprocity law for division fields of Drinfeld modules proved in our earlier paper. We then use 
		these results to give an efficient algorithm for computing the endomorphism rings and discuss some 
		interesting examples produced by our algorithm. 
	\end{abstract}
	
	\maketitle

%---------------------------------------------

\section{Introduction} 

\subsection{Drinfeld modules}
We first recall some basic concepts from the theory of Drinfeld modules. 

Let $\F_q$ be a finite field with $q$ elements and $A=\F_q[T]$ be the polynomial ring over $\F_q$ in an indeterminate $T$. 
Let $F=\F_q(T)$ be the field of fractions of $A$. We will call a nonzero prime ideal of $A$ simply a \textit{prime} of $A$. 
Given a prime $\fp$ of $A$, we denote by $A_\fp$ (resp. $F_\fp$) the completion of $A$ at $\fp$ (resp. the field of 
fractions of $A_\fp$). %We denote by $\F_\fp=A/\fp$ the residue field of $A_\fp$. 
By abuse of notation we will denote the monic generator of $\fp$ by the same symbol. 

Let $L$ be a field equipped with a structure $\gamma: A\to L$ of an 
$A$-algebra. Let $\tau$ be the Frobenius 
endomorphism of $L$ relative to $\F_q$, that is, the map $\alpha\mapsto \alpha^q$. Let $L\{\tau\}$ 
be the noncommutative ring of polynomials in $\tau$ with coefficients in $L$ 
and the commutation rule $\tau c=c^q\tau$, $c\in L$. 
A \textit{Drinfeld module of rank $r\geq 1$ defined $L$} is a ring homomorphism 
$\phi: A \to L\{\tau\} $, $a\mapsto\phi_a$, 
uniquely determined by the image of $T$
$$
\phi_T=\gamma(T)+\sum_{i=1}^{r} g_i(T) \tau^i,\quad g_r(T)\neq 0. 
$$
%A \textit{morphism} from the Drinfeld module $\phi$ to the Drinfeld module $\psi$ over $L$ 
%is some $u\in L\{\tau\}$ such that $u\phi_a=\psi_a u$ for all $a\in A$ (it suffices to require this 
%for $a=T$); $u$ is an \textit{isomorphsim} if $u\in L\{\tau\}^\times=L^\times$.  With this definition, 

The \textit{endomorphism ring} of $\phi$ is the centralizer of $\phi(A)$ in $L\{\tau\}$:
$$
\End_L(\phi)=\{u\in L\{\tau\}\mid u\phi_T=\phi_T u\}. 
$$
It is clear that $\End_L(\phi)$ contains in its center 
the subring $\phi(A)\cong A$, hence is an $A$-algebra. 
It can be shown that $\End_L(\phi)$ is a free $A$-module of rank $\leq r^2$; see \cite{Drinfeld}. 

The Drinfeld module $\phi$ endows the algebraic closure $\overline{L}$ of $L$ with an $A$-module structure, where $a\in A$ 
acts by $\phi_a$. The \textit{$a$-torsion} $\phi[a]\subset \overline{L}$ of $\phi$ is the kernel of $\phi_a$, that is, the set of zeros 
of the polynomial 
$$
\phi_a(x)=\gamma(a)x+\sum_{i=1}^{r\cdot \deg(a)}g_i(a)x^{q^i}\in L[x]. 
$$
It is easy to see that $\phi[a]$ 
has a natural structure of an $A$-module, where $A$ acts via $\phi$. Moreover,  
if $a$ is not divisible by $\ker(\gamma)$, then 
$\phi[a]\cong (A/aA)^{\oplus r}$ and $\phi[a]$ is contained in the separable closure $L^\sep\subset \overline{L}$ (since $\phi'_a(x)=\gamma(a)\neq 0$). 
%On the other hand, if $a=\ker(\gamma)$, then $\phi[a]\cong (A/aA)^{\oplus r'}$ for some $0\leq r'<r$. 

For a prime $\fl\lhd A$ 
different from $\ker(\gamma)$, the \textit{$\fl$-adic Tate module of $\phi$} is the inverse limit    
$$T_\fl(\phi)=\underset{\longleftarrow}{\lim}\ \phi[\fl^n]\cong A_\fl^{\oplus r}.$$ 

\subsection{Main results}\label{ssMR}
Let $\fp\lhd A$ be a prime and $k$ a finite extension of $\F_\fp:=A/\fp$. We consider $k$
as an $A$-algebra via the composition $\gamma: A\to A/\fp\hookrightarrow k$. %The $A$-characteristic of $k$ is $\fp=\ker(\gamma)$. 
Let $\phi$ be a Drinfeld module of rank $r$ defined over $k$. 
Denote $\cE=\End_k(\phi)$ and $D=\cE\otimes_A F$. 
Let $\pi:=\tau^{[k:\F_q]}\in k\{\tau\}$. It is clear that $\pi$ is in the center of $k\{\tau\}$. In particular, $\pi$ commutes with $\phi(A)$, so  
$\pi\in \cE$. Let $K=F(\pi)$ be the subfield of $D$ generated by $\pi$. The following is well-known 
(cf. \cite{Drinfeld2}, \cite{GekelerFDM}): % [(2.3), (2.9)]
\begin{itemize}
	\item The degree of the field extension $K/F$ divides $r$. 
	\item $D$ is a central division algebra over $K$ of dimension $\left(r/[K:F]\right)^2$. 
		\item There is a unique place of $K$ over the place $\infty=1/T$ of $F$. 
\end{itemize}

The endomorphism rings (and algebras) of Drinfeld modules 
over finite fields have been extensively studied; cf. \cite{Drinfeld2}, \cite{Angles}, \cite{GekelerFDM}, \cite{GekelerTAMS}, \cite{JKYu}. 
They play an important role in the theory of Drinfeld modules, as well as their applications to other areas, such as 
the theory of central simple algebras (cf. \cite{GekelerCR}), the Langlands conjecture over function fields (cf. \cite{Drinfeld2}, \cite{LaumonCDV}), 
or the study of the splitting behavior of primes in certain non-abelian extensions of $F$ (cf. \cite{CP}, \cite{GP}, \cite{CS2}). 
In this paper, we are interested in comparing $\cE$ to $A[\pi]$. We then deduce from this a method for computing $\cE$. 

Throughout the paper, we make the following assumption: 
\begin{equation}\label{eq-assumption}
[K:F]=r.
\end{equation}
This assumption is satisfied if, for example, 
$k=\F_\fp$ (cf. \cite[Prop. 2.1]{GP}) or $\phi[\fp]\cong (A/\fp)^{r-1}$, i.e., $\phi$ is ordinary 
(cf. \cite[(2.5)]{LaumonCDV}). It is equivalent to the assumption that the endomorphism algebra $D$ is commutative, or more precisely, 
that $\cE$ is an $A$-order in $K$. In that case, $A[\pi]\subset \cE$ are $A$-orders in $K$, so by the theory of finitely generated modules 
over principal ideal domains we have 
$$
\cE/A[\pi]\cong A/b_1A\times A/b_2A\times \cdots\times A/b_{r-1}A
$$
for uniquely determined nonzero monic polynomials $b_1, \dots, b_{r-1}\in A$ such that 
$$
b_1\mid b_2\mid \cdots \mid b_{r-1}. 
$$
We call the $(r-1)$-tuple $(b_1, \dots, b_{r-1})$ the \textit{Frobenius index} of $\phi$. The first main result of this paper is 
the following: 

\begin{thm}\label{thmFrobIndex}
	For any $1\leq i\leq r-1$, there is a monic 
	polynomial $f_i(x)\in A[x]$ of degree $i$ such that $f_i(\pi)\in b_i\cE$. Moreover, if there is a 
	monic polynomial $g(x)\in A[x]$ of degree $i$ and $b\in A$ such that $g(\pi)\in b\cE$ 
	then $b$ divides $b_i$. 
\end{thm}

The proof of this theorem is given in Section \ref{sEndom}. It is based on the existence of a special type bases of $A$-orders;  
this crucial fact about orders is discussed separately in Section \ref{sOrders}. 
As we explain in Remark \ref{rem3.2}, Theorem \ref{thmFrobIndex} can be 
considered as a refinement of the reciprocity law proved in our earlier paper \cite[Thm. 1.2]{GP} (see also \cite[Cor. 2]{CP} for the rank-$2$ case).   

The condition $f_i(\pi)\in b_i\cE$ means that we have an equality $f_i(\pi)=u\phi_{b_i}$ in $k\{\tau\}$ for some 
$u\in \cE$. For given $b_i$ and $f_i$, the validity of the equality 
$f_i(\pi)=u\phi_{b_i}$ can be easily checked using the division algorithm in $k\{\tau\}$. On the other hand, 
a finite list of possible Frobenius indices of $\phi$ can be deduced either by computing the discriminant of $A[\pi]$, or 
by computing an $A$-basis of the integral closure of $A$ in $K$. 
Since we can assume that the coefficients of $f_i(x)\in A[x]$ have degrees less than the degree of $b_i$, the Frobenius 
index of $\pi$ can be determined by performing finitely many calculations. This leads to an efficient algorithm 
for computing the Frobenius index of $\phi$ and an $A$-basis of $\cE$. The algorithm is described in detail in Section \ref{sEndom}. 
We have implemented this algorithm in \texttt{Magma} 
software package.  In Section \ref{sEndom}, the reader will find 
an explicit example of a calculation of the endomorphism ring of a Drinfeld module of rank $3$. 
In the rank-$2$ case, we gave another algorithm for computing $\cE$ in \cite{GP}; the present algorithm 
is different from the one in \cite{GP}, even when specialized to $r=2$. 

A completely different 
algorithm from ours for computing the endomorphism rings of Drinfeld modules was proposed by Kuhn and Pink in \cite{KP}. 
This algorithm works in all cases, without the restriction \eqref{eq-assumption}, and determines a basis of $\cE$ as an $\F_q[\pi]$-module. 
However, it is not quite clear how easily one can deduce from this the $A$-module structure of $\cE$, e.g., determine the Frobenius 
index or the discriminant of $\cE$ over $A$.  We discuss the algorithm of Kuhn and Pink in more detail in Remark \ref{remKP}. 

Next, we explain a theoretical application of Theorem \ref{thmFrobIndex}. 
Let $\Phi: A\to F\{\tau\}$ be a Drinfeld module of rank $r$ over $F$ defined by 
$$
\Phi_T=T+g_1\tau+\cdots+g_r\tau^r. 
$$
(We will always implicitly assume that $\gamma: A\to F$ is the canonical embedding of $A$ into its field of fractions.)
We say that a prime $\fp\lhd A$ is a \textit{prime of good reduction} for $\Phi$ if $\ord_\fp(g_i)\geq 0$ for $1\leq i\leq r-1$, 
and $\ord_\fp(g_r)=0$. (All but finitely many primes 
of $A$ are primes of good reduction for a given Drinfeld module $\Phi$.) 
Let $n\in A$ be a monic polynomial and 
$F(\Phi[n])$ be the splitting field of the polynomial $\Phi_n(x)$; such fields are called \textit{division fields} (or \textit{torsion fields}) of $\Phi$.  
If $\fp$ is a prime of good reduction of $\Phi$ and $\fp\nmid n$, then $\fp$ does not ramify in $F(\Phi[n])$. 
One is interested in the splitting behavior of $\fp$ in $F(\Phi[n])$ as $n$ varies, e.g., 
a ``reciprocity law'' in the form of  
congruence conditions modulo $n$ which guarantee that $\fp$ splits completely $F(\Phi[n])$. 
The primes which split completely in $F(\Phi[n])$ have been studied before, 
e.g. \cite{CS2}, \cite{GP}, \cite{KL}.

We can consider $g_1, \dots, g_r$ as elements of $A_\fp$; denote by 
$\overline{g}$ the image of $g\in A_\fp$ under the canonical homomorphism $A_\fp\to A_\fp/\fp=\F_\fp$. 
The \textit{reduction of $\Phi$ at $\fp$} is the 
Drinfeld module $\phi$ over $\F_\fp$ given by 
$$
\phi_T=\overline{T}+\overline{g_1}\tau+\cdots+\overline{g_r}\tau^r.  
$$
Note that $\phi$ has rank $r$ since $\overline{g_r}\neq 0$. Let $\cE=\End_{\F_\fp}(\phi)$ and $\pi=\tau^{\deg(\fp)}$. 
We have the inclusion of orders $A[\pi]\subset \cE$. Theorem \ref{thmFrobIndex}, or rather its proof, provides an 
explicit basis of $\cE$ as a free $A$-module of rank $r$. With respect to this basis, the action of $\pi$ on $\cE$ by multiplication can be described 
by an explicit matrix $\cF(\fp)\in \Mat_r(A)$ which depends on the Frobenius index of $\phi$, the coefficients 
of the polynomials $f_i$, and the coefficients of the minimal polynomial of $\pi$. We explain in Section \ref{sFrob} that 
under a mild (but subtle) technical assumption on $\cE$, the 
integral matrix $\cF(\fp)$, when reduced modulo $n$, represents the conjugacy class of the Frobenius at $\fp$ 
in $\Gal(F(\Phi[n])/F)\subset \GL_r(A/nA)$. This result in the rank-2 case was proved in \cite[Thm. 1]{CP}. 
The analogue of this result for elliptic curves goes back to Duke and T\'oth \cite{DT}, which was our initial motivation 
for considering this problem in the setting of Drinfeld modules. (For a refinement of the result of Duke and T\'oth 
for elliptic curves see the paper by Centeleghe \cite{Centeleghe}.) 

The technical assumption mentioned in the previous paragraph is the assumption that $\cE\otimes_A A_\fl$ 
is a Gorenstein ring for all primes $\fl\mid n$. It is often satisfied (see Proposition \ref{propGor}), but not always 
when $r\geq 3$. At the end of Section \ref{sFrob} we give an interesting example of $\cE\otimes_A A_\fl$ which 
is not Gorenstein. The study of Gorenstein property of endomorphism rings of abelian varieties, especially the Jacobian varieties 
of modular curves, has played an important role in many fundamental developments in arithmetic geometry (cf. \cite{Mazur}), 
but, as far as we are aware, it has not been studied at all in the context of Drinfeld modules.

%---------------------------------------------

\section{Some facts about orders}\label{sOrders}

Let $A$ be a principal ideal domain with field of fraction $F$. Let $f(x)\in A[x]$ be a monic irreducible 
polynomial of degree $r$ with coefficients in $A$. Fix a root $\pi$ of $f(x)$ in $\overline{F}$ 
and denote $K=F(\pi)\subset \overline{F}$. Let $B$ be the integral closure of $A$ in $K$. 

The field $K$ is an $r$-dimensional vector space over $F$. For a given $\alpha\in K$, multiplication by $\alpha$ on $K$ 
defines an $F$-linear transformation $M_\alpha: K\to K$. Let $\Tr_{K/F}(\alpha)\in F$ (resp. norm $\Nr_{K/F}(\alpha)\in F$) 
be the trace (resp. determinant) of $M_\alpha$. The \textit{discriminant} of an $r$-tuple $\alpha_1, \dots, \alpha_r\in K$ is 
$$\disc(\alpha_1, \dots, \alpha_r) =\det\left(\Tr_{K/F}(\alpha_i\alpha_j)\right).
$$ 
The discriminant does not depend of the ordering of the elements $\alpha_j$. 
It is known that $\disc(\alpha_1, \dots, \alpha_r) =0$ if and only if either $K/F$ is inseparable or $\alpha_1, \dots, \alpha_r$ are linearly dependent over $F$. 
Moreover, (cf. \cite[Ch. III]{SerreLF}): % or \cite[Ex. VI.32]{Lang}):
\begin{enumerate}
	\item[(i)] $\disc(1, \pi, \dots, \pi^{n-1}) =(-1)^{n(n-1)/2} \Nr_{K/F}(f'(\pi))$. 
	\item[(ii)] If $\begin{bmatrix} \beta_1\\ \vdots \\ \beta_r\end{bmatrix}=M \begin{bmatrix} \alpha_1\\ \vdots \\ \alpha_r\end{bmatrix}$ with $M\in \Mat_r(F)$, then 
	$$
	\disc(\beta_1, \dots, \beta_r)=\det(M)^2\cdot \disc(\alpha_1, \dots, \alpha_r). 
	$$
\end{enumerate}

An \textit{$A$-order} in $K$ is an $A$-subalgebra $\cO$ of $B$ with the same unity element and such that $B/\cO$ has finite cardinality. 
Note that any $A$-order $\cO$ is a free $A$-modules of rank $r$. An example of an $A$-order in $K$ is $$A[\pi]=A+A\pi+\cdots+A\pi^{r-1}.$$ 
Let $\cO\subset \cO'$ be two $A$-orders in $K$. Since both modules $\cO$ and $\cO'$ have the same rank over $A$, and both contain $1$, 
we have
$$
\cO'/\cO\cong A/b_1 A\times A/b_2 A\times \cdots \times A/b_{r-1}A,
$$
for uniquely determined (up to multiplication by units in $A$) non-zero elements $b_1, \dots, b_{r-1}\in A$ such that 
$$
b_1\mid b_2\mid \cdots \mid b_{r-1}. 
$$
This is an easy consequence of the theory of finitely generated modules over principal ideal domains; cf. \cite[Thm. 12.5]{DF}. 
We call the ideal $\chi(\cO'/\cO)$ of $A$ generated by $\prod_{i=1}^{r-1}b_i$ 
the \textit{index} of $\cO$ in $\cO'$, and $(b_1, \dots, b_{r-1})$ the \textit{refined index} of $\cO$ in $\cO'$ 
(in a more standard terminology, the elements $b_1, \dots, b_{r-1}\in A$ are the \textit{invariant factors} of $\cO'/\cO$).   

Let $\cO$ be an $A$-order in $K$. Let $\alpha_1, \dots, \alpha_r$ be a basis of $\cO$ over $A$: 
$$
\cO=A\alpha_1+\cdots+A\alpha_r. 
$$
Define the discriminant $\disc(\cO)$ of $\cO$ to be the ideal of $A$ generated by 
$\disc(\alpha_1, \dots, \alpha_r)$.
By (ii) above, $\disc(\cO)$ does not depend on the choice of a basis $\alpha_1, \dots, \alpha_r$. Moreover, 
by the same property (see also \cite[Ch. III]{SerreLF}), if $\cO\subseteq \cO'$ is an inclusion of orders, then 
\begin{equation}\label{eq-discchi}
\disc(\cO)=\chi(\cO'/\cO)^2\cdot \disc(\cO'),
\end{equation}
and for an inclusion of orders $\cO\subseteq \cO'\subseteq \cO''$  we have 
\begin{equation}\label{eq-chichi}
\chi(\cO''/\cO)=\chi(\cO''/\cO')\cdot \chi(\cO'/\cO). 
\end{equation}

The following theorem is essentially Theorem 13 in \cite{Marcus}. Since this fact is crucial for our later purposes and we 
need the statement in a more general setting than in \cite{Marcus}, for the sake of completeness we give the proof. 

\begin{thm}\label{thm2}
Assume $\cO$ is an $A$-orders in $K$ such that $A[\pi]\subset \cO$. Let $(b_1, \dots, b_{r-1})$ be the refined 
index of $A[\pi]$ in $\cO$. There are polynomial $f_i(x)\in A[x]$, $1\leq i\leq r-1$, such that $f_i$ is monic, $\deg(f_i)=i$, and  
$$
1, \frac{f_1(\pi)}{b_1}, \cdots, \frac{f_{r-1}(\pi)}{b_{r-1}}, 
$$   
is an $A$-basis of $\cO$. 
\end{thm}
\begin{proof} Fix a generator $d$ of the ideal $\chi(\cO/A[\pi])$. 
For each $k$, $1\leq k\leq r$, let $G_k$ be the free $A$-submodule of $K$ generated by $1/d, \pi/d, \dots, \pi^{k-1}/d$. 
Let $\cO_k=\cO\cap G_k$. 
Note that $\cO_1=A$, $\rank_A \cO_k=k$ (since $A+A\pi+\cdots+A\pi^k\subset \cO_k$) and $\cO_n=\cO$ 
(since for any $\alpha\in \cO$ we have $d\alpha\in A[\pi]$). 

We will define $d_1\mid d_2\mid \cdots\mid d_{n-1}\mid d$ and monic  
polynomials $f_i(x)\in A[x]$ of degree $i$, $1\leq i\leq r-1$, such that for each $1\leq k\leq r$
$$1, f_1(\pi)/d_1, \dots, f_{k-1}(\pi)/d_{k-1}$$ is an $A$-basis 
of $\cO_k$. This is certainly true for $k=1$. Now fix some $1\leq k< r$ and assume we were able to prove this claim for $\cO_{k}$. 
Let 
$$
\eta: G_{k+1}\To A, \qquad \sum_{i=1}^{k} a_i \frac{\pi^i}{d}\longmapsto a_{k},
$$
be the projection onto the last factor. 
Let $I=\eta(\cO_{k+1})$ be the image of $\cO_{k+1}$ under this homomorphism. Clearly $I$ is an ideal of $A$, thus can be 
generated by a single element. Fix some $\beta\in \cO_{k+1}$ such that $\eta(\beta)$ generates $I$. Note that $I\neq 0$ (since $\pi^k\in \cO_{k+1}$ 
maps to $d\neq 0$), and 
$\cO_{k}\subseteq \ker(\eta|\cO_{k+1})$. Since $I$ is a free $A$-module of rank $1$, comparing the ranks we conclude that 
$\cO_{k}=\ker(\eta|\cO_{k+1})$ and $\cO_{k+1}=\cO_k\oplus A\beta$. It remains to show that $\beta=f_k(\pi)/d_k$. We have 
$$
\frac{d}{d_{k-1}}=\eta\left(\frac{\pi f_{k-1}(\pi)}{d_{k-1}}\right)
$$
and this is in $I$. It follows that $a \eta(\beta)=d/d_{k-1}$ for some $a\in A$. Defining $d_k=a d_{k-1}$, we have $\eta(\beta)=d/d_k$, 
which implies that $\beta=f_k(\pi)/d_k$ for some $f_k(x)=x^k+$lower degree terms. Note that by construction $d_{k-1}\mid d_k\mid d$, 
so it remains to show that the coefficients of $f_k(x)$ are in $A$. However, since $f_k(\pi)/d_{k-1}=a\beta\in \cO_{k+1}$, we have 
$$
\frac{f_k(\pi)-\pi f_{k-1}(\pi)}{d_{k-1}}=:\gamma\in \cO_{k+1}. 
$$
On the other hand, $\eta(\gamma)=a\eta(\beta)-d/d_{k-1}=0$, so in fact, $\gamma\in \cO_k$. Using our basis for $\cO_k$ 
we can write $\gamma=g(\pi)/d_{k-1}$ for some $g(x)\in A[x]$ having degree $<k$. This implies that $f_k(\pi)-\pi f_{k-1}(\pi)=g(\pi)$. 
Since the degree of minimal polynomial of $\pi$ over $F$ is $r$ and the degree of $f_k(x)-\pi f_{k-1}(x)-g(x)$ is strictly less than $r$, 
we must have $f_k(x)=x f_{k-1}(x)+g(x)\in A[x]$. 

It remains to show that $(d_1, \dots, d_{r-1})$ is the refined index of $A[\pi]$ in $\cO$. Since the polynomials 
$f_i(x)$ are monic, the elements $1, f_1(\pi), \dots, f_{n-1}(\pi)$ form an $A$-basis of $A[\pi]$ (the transition matrix from 
$1, \pi, \dots, \pi^{n-1}$ to this basis is upper triangular with $1$'s on the diagonal, so has determinant $1$).  But now, using 
$1, f_1(\pi)/d_1, \dots, f_{n-1}(\pi)/d_{n-1}$ as a basis of $\cO$, we obviously have 
$\cO/A[\pi]\cong A/d_1A\times \cdots \times A/d_{r-1}A$ with $d_1\mid \cdots \mid d_{r-1}$. Since the invariant factors of $\cO/A[\pi]$ 
are unique, up to multiplication by units, $d_1, d_2, \dots, d_{r-1}$ must be the invariant factors. 
\end{proof}

\begin{rem}
	The polynomials $f_i\in A[x]$, $1\leq i\leq r-1$, in Theorem \ref{thm2} are not unique. It is easy to see that they 
	can be replaced by any other monic 
	polynomials $g_i\in A[x]$ such that $g_i$ has degree $i$ and all $g_i(\pi)/b_i$ are in $\cO$. 
\end{rem}

\begin{prop}\label{thm3}
	Let $\cO$ is an $A$-orders in $K$ such that $A[\pi]\subset \cO$. Let $(b_1, \dots, b_{r-1})$ be the refined 
	index of $A[\pi]$ in $\cO$.
	\begin{enumerate} 
		\item If $i+j<r$, then $b_ib_j\mid b_{i+j}$. 
	\item For any $i<r$, we have $b_1^i\mid b_i$.  
	\item $b_1^{n(n-1)}\mid \disc(A[\pi])$. 
	\item If $b_1\neq 1$, then the inclusion of ideal $(b_1)\supsetneq (b_2)\supsetneq\cdots\supsetneq (b_{r-1})$ 
	is strict. 
	\end{enumerate}
\end{prop}
\begin{proof} Let $f_1(\pi)/b_1, \dots, f_{r-1}(\pi)/b_{r-1}$ be the $A$-basis of $\cO$ supplied by Theorem \ref{thm2}. 
	Consider $\alpha=f_i(\pi)f_j(\pi)/b_ib_j \in \cO$. We can express $\alpha$ as an $A$-linear combination  
	$\alpha=a_0+\sum_{k=1}^{i+j} a_kf_k(\pi)/b_k$. 
	(The basis elements $f_k(\pi)/b_k$ for $k>i+j$ do not appear in this linear combination since otherwise $\pi$ 
	would a root of a non-zero polynomial of degree $<r$.) 
	Comparing the coefficients of $\pi^{i+j}$ on both sides we get $1/b_ib_j=a_{i+j}/b_{i+j}$ for some nonzero $a_{i+j}\in A$. 
	Thus, $b_ib_j$ divides $b_{i+j}$. This proves (1). 
	
	(2) and (4) immediately follow from (1).  Next, by \eqref{eq-discchi}, 
	$\chi(\cO/A[\pi])^2=(b_1\cdots b_{n-1})^2$ 
	divides $\disc(A[\pi])$. On the other hand, $b_1^i$ divides $b_i$, so 
	$b_1^{2(1+2+3+\cdots+(n-1))}=b_1^{n(n-1)}$ divides $\disc(A[\pi])$. This proves (3). 
\end{proof}

%------------------------------------------------

\section{Endomorphism rings of Drinfeld modules}\label{sEndom} 
Let the notation and assumptions be as at the beginning of Section \ref{ssMR}. In particular, $\phi$ 
is a Drinfeld module of rank $r$ over a finite extension $k$ of $\F_\fp$, $\cE=\End_k(\phi)$, and $A[\pi]\subset \cE$ 
is the suborder generated by the Frobenius endomorphism of $\phi$. As in Section \ref{ssMR}, assume \eqref{eq-assumption}. 

 \begin{thm}\label{thmFrobIndex2}
 	Let $(b_1, \dots, b_{r-1})$ be the Frobenius index of $\phi$. For any $1\leq i\leq r-1$, there is a monic 
 	polynomial $f_i(x)\in A[x]$ of degree $i$ such that $f_i(\pi)\in b_i\cE$. Moreover, if there is a 
 	monic polynomial $g(x)\in A[x]$ of degree $i$ and $b\in A$ such that $g(\pi)\in b\cE$ 
 	then $b$ divides $b_i$. 
 \end{thm}
\begin{proof} Theorem \ref{thm2}, applied to $A[\pi]\subset \cE$, implies the existence of monic polynomials $f_i$ of degree $i$, $1\leq i\leq r-1$, 
	such that $f_i(\pi)\in b_i\cE$. 
	
	Now assume there is a monic polynomial $g(x)\in A[x]$ of degree $i$ and $b\in A$ such that $g(\pi)\in b\cE$.  
	Suppose there is a prime $\fq\in A$ such that $\fq\mid b$ but $\fq\nmid b_i$. Then we can find $z_1, z_2\in A$ such that 
	$z_1 b_i+z_2\fq=1$. The polynomial $z_1b_i g(x)+z_2\fq f_i(x)\in A[x]$ is monic of degree $i$, and 
	$(z_1b_i g(\pi)+z_2\fq f_i(\pi))/\fq b_i\in \cE$. Since the largest exponent of $\pi$ in $z_1b_i g(\pi)+z_2\fq f_i(\pi)$ is $i$, there 
	exist $a_0, \dots, a_i\in A$ such that 
	$$
	\frac{z_1b_i g(\pi)+z_2\fq f_i(\pi)}{\fq b_i}=a_0+a_1\frac{f_1(\pi)}{b_1}+\cdots+a_i\frac{f_i(\pi)}{b_i}. 
	$$
	Multiplying both sides by $b_i\fq$, we get an equation in $A[\pi]$ where the left hand side is a monic polynomial in $\pi$ of degree $i$, while 
	the right hand side has degree $i$ in $\pi$ and leading coefficient $a_i\fq$. 
	This implies that $\pi$ satisfies a polynomial in $A[x]$ of degree less than $r$, contradicting 
	\eqref{eq-assumption}. Hence every prime divisor of $b$ is also a divisor of $b_i$. Write 
	$b=x_1y_1$ and $b_i=x_2y_2z$, where $x_1$ and $x_2$ have the same prime divisors, $y_1$ and $y_2$ 
	have the same prime divisors, $x_2\mid x_1$, $y_1\mid y_2$, and $\gcd(x_1, y_1)=\gcd(x_2, y_2)=\gcd(x_2, z)=\gcd(y_2, z)=1$. 
	As earlier, since $g(\pi)/b\in \cE$, we can write 
	$$
	\frac{g(\pi)}{x_1y_1}=a_0+a_1\frac{f_1(\pi)}{b_1}+\cdots+a_i\frac{f_i(\pi)}{x_2y_2z}. 
	$$
	After multiplying both sides by $x_1y_2z$, the coefficient of $\pi^i$ on the left hand side of the resulting equation is $y_2z/y_1$, 
	whereas on the right hand side the corresponding coefficient is $a_ix_1/x_2$. Since $x_1/x_2$ is coprime to $y_2z/y_1$, 
	$x_1$ and $x_2$ must be equal, up to multiplicative units. Since $y_1\mid y_2$, we see that $b$ divides $b_i$. 
\end{proof}

\begin{rem}\label{rem3.2} The condition $f_i(\pi)\in b_i\cE$ means that we have an equality $f_i(\pi)=u\phi_{b_i}$ in $k\{\tau\}$ for some 
	$u\in \cE$. From this it is obvious that $f_i(\pi)$ acts as $0$ on $\phi[b_i]$. Conversely, it is not hard to prove that if 
	$b$ is coprime to $\fp$ and $g(\pi)$ acts as $0$ on $\phi[b]$, then $g(\pi)\in b\cE$; see the proof of Theorem 1.2 in \cite{GP}.  Hence the 
	previous theorem essentially says that $b_i\in A$ is the element of largest degree such that $\pi$ acting on $\phi[b_i]$ 
	satisfies a polynomial of degree $i$, whereas the minimal polynomial of $\pi$ acting on any Tate module $T_\fl(\phi)$ has degree $r$. 
	
	Now suppose $\phi$ is the reduction at $\fp$ of a Drinfeld module $\Phi$ over $F$. Let $n\in A$ 
	be a polynomial not divisible by $\fp$. Assume $r$ is coprime to the characteristic of $F$. In \cite{GP}, we proved a reciprocity law 
	which says that  $\fp$ splits completely in the Galois extension $F(\Phi[n])$ of $F$ if and only if $n$ divides both $b_1$ and 
	$a_{r-1}+r$, where $a_{r-1}$ is the coefficient of $x^{r-1}$ in the minimal polynomial of $\pi$. The starting point of the 
	proof of this result is the observation that we have an isomorphism $\Phi[n]\cong \phi[n]$ compatible with the action of the Frobenius at $\fp$ 
	on $\Phi[n]$ and the action of $\pi$ on $\phi[n]$. Then the proof proceeds by showing that $\pi$ acts as a scalar 
	on $\phi[n]$ if and only if $n\mid b_1$. 
As follows from the previous paragraph, this last fact is a special case of Theorem \ref{thmFrobIndex2}. Thus, 	
Theorem \ref{thmFrobIndex2} is a refinement of our reciprocity law in the sense that we give a Galois-theoretic 
interpretation of all $b_i$'s, not just $b_1$. Moreover, as we will see Section \ref{sFrob}, 
$b_1, \dots, b_{r-1}$ appear in a 
matrix representing the Frobenius at $\fp$ in $\Gal(F(\Phi[n])/F)\subseteq \GL_r(A/nA)$. 
\end{rem}

Theorem \ref{thmFrobIndex2} can be used to give an efficient algorithm for 
computing the Frobenius index of a Drinfeld module and an $A$-basis of its endomorphism ring. 
The algorithm has two main steps. 

\vspace{0.1in}

\textbf{Step 1.} Let $\phi$ be a Drinfeld module of rank $r$ over $k$ given in terms of $\phi_T\in k\{\tau\}$. 
Let $B$ be the integral closure of $A$ in $K$. 

Start by computing the minimal polynomial $P(x)\in A[x]$ of $\pi$ over $A$.  
(Note that under our assumption \eqref{eq-assumption}, $P(x)$ is also the 
characteristic polynomial of the Frobenius automorphism $\alpha\mapsto \alpha^{\# k}$ 
of $\Gal(\bar{k}/k)$ acting on $T_\fl(\phi)$ for any $\fl\neq \fp$.)  
When $k=\F_\fp$, computing $P(x)$ is relatively easy; see \cite[$\S$5.1]{GP}. When $[k:\F_\fp]>1$, this calculation becomes 
more involved. An algorithm for computing $P(x)$ for $r=2$ and arbitrary $k$ is described in \cite[$\S$3]{GekelerTAMS}.  

Next, compute the index $\chi(B/A[\pi])$. There are known algorithms for computing a basis of the integral closure of 
$A$ in a field extension of $F$ given by an explicit polynomial (the polynomial in our case 
is $P(x)$); such an algorithm is implemented in \texttt{Magma}. 
The index $\chi(B/A[\pi])$ can be computed by expressing $\pi$ in a given $A$-basis of $B$. Alternatively,  
if $K/F$ is separable, then $\chi(B/A[\pi])$ can be computed from $\disc(A[\pi])$ and $\disc(B)$, 
since by \eqref{eq-discchi} 
$$\disc(A[\pi])=\chi(B/A[\pi])^2\cdot\disc(B).$$
(In fact, for our purposes, it is enough to have an upper bound on $\chi(B/A[\pi])$ which is already provided by $\disc(A[\pi])$.)
%The discriminant of $A[\pi]$ is generated by the discriminant of 
%As for $\disc(B)$, there are known algorithms for computing the 
%discriminant of the integral closure of $A$ in a separable field extension of $F$ given by an explicit polynomial (the polynomial in our case 
%is of course $P(x)$); such an algorithm is implemented in \texttt{Magma}. 
%$P(x)$ will serve as the irreducible polynomial defining $K$. 

Having computed the index $\chi(B/A[\pi])$, one can produce a finite list of possible Frobenius indices $(b_1, \dots, b_{r-1})$.    
Indeed, by \eqref{eq-chichi},
$$
\chi(B/A[\pi])=\chi(B/\cE)\cdot  \chi(\cE/A[\pi]), 
$$
and $\chi(\cE/A[\pi])=(\prod_{i=1}^{r-1}b_i)$ divides $\chi(B/A[\pi])$. 
We get further constrains on possible $(b_1, \dots, b_{r-1})$ from the divisibilities (cf. Proposition \ref{thm3})
\begin{align*}
& b_i\mid b_j, \quad 1\leq i<j\leq r-1,\\
& b_ib_j\mid b_{i+j}, \quad i+j<r, \\
& \text{if $0<\deg(b_1)$, then $\deg(b_1)<\deg(b_2)<\cdots< \deg(b_{r-1})$},\\
& b_1^{r(r-1)}\mid \disc(A[\pi]). 
\end{align*}
We arrange all possible $(b_1, \dots, b_{r-1})$ by the degrees of products $\prod_{i=1}^{r-1}b_i$, from the highest to zero. 

\vspace{0.1in}

\textbf{Step 2.} Starting with the first entry in our list of possible $(b_1, \dots, b_{r-1})$, check if for all $i=1, \dots, r-1$ there are 
polynomials $f_i(x)\in A[x]$, $\deg_x(f_i(x))=i$, such that $f_i(\pi)\in b_i\cE$. 

Given a polynomial 
$$
g(x)=x^s+a_{s-1}x^{s-1}+\cdots +a_0,
$$
checking whether $g(\pi)\in b\cE$ can be done as follows. First, compute the residue of 
$\pi^s+\phi_{a_{s-1}}\pi^{s-1}+\cdots+\phi_{a_0}$ modulo $\phi_{b}$ using the right division 
algorithm in $k\{\tau\}$. If the residue is nonzero, then $g(\pi)\not\in b\cE$. If the residue is $0$, then 
$g(\pi)=u\phi_{b}$ for an explicit $u\in k\{\tau\}$ produced by the devision algorithm. Now check 
if  the commutation relation $u\phi_T=\phi_Tu$ holds in $k\{\tau\}$; this relation holds if and only if $u\in \cE$. 

Since we can assume that the coefficients of $f_i(x)\in A[x]$  
have degrees $< \deg(b_i)$ (as polynomials in $T$), there are only finitely many possibilities for $f_i(x)$.  
If for some possible choice of $f_1, \dots, f_{r-1}$ we have $f_i(\pi)\in b_i\cE$, then $(b_1, \dots, b_{r-1})$ 
is the Frobenius index of $\phi$. If none of the choices of $f_1, \dots, f_{r-1}$ work, then $(b_1, \dots, b_{r-1})$  
is not the Frobenius index and we move to the next possible Frobenius index. Since one of $(b_1, \dots, b_{r-1})$'s 
is the actual Frobenius index, this step will eventually find it. 
(There can be several ``candidate'' Frobenius indices satisfying the necessary condition of this step, i.e., the existence of $f_i$'s. 
One can distinguish the actual Frobenius index  among these ``candidate'' Frobenius indices 
using the maximality property of Frobenius indices given by Theorem \ref{thmFrobIndex}. Since we have arranged the list 
of possible Frobenius indices by decreasing degrees of $\prod_{i=1}^{r-1}b_i$, we always find the actual Frobenius index first.)

Finally, having determined the Frobenius index $(b_1, \dots, b_{r-1})$ of $\phi$ and the polynomials $f_1, \dots, f_{r-1}$ 
such that $f_i(\pi)\in b_i\cE$, 
we compute an explicit $A$-basis of $\cE$  in $k\{\tau\}$ by dividing $f_i(\pi)$ by $\phi_{b_i}$ using the division algorithm for $k\{\tau\}$. 
%The basis is $\{f_1(\pi)/\phi_{b_1}, \dots, f_{r-1}(\pi)/\phi_{b_{r-1}}\}$. 

\vspace{0.1in} 

We have implemented the above algorithm in \texttt{Magma} and computed the Frobenius indices 
and bases of endomorphism rings for various Drinfeld modules of rank $r=2$ and $3$. For rank $2$  
this algorithm corroborates the data found by our previous (different) algorithm \cite{GP}. For rank $3$, we have found 
some examples where $b_1$ and $b_2$ have positive degrees. 

\begin{example}\label{example3.3}
	Let $q=5$, $r=3$, $\fp=T^6+3T^5+T^2+3T+3$, and $k=\F_\fp$. Let $\phi:A\to \F_\fp\{\tau\}$ be given by 
	$$
	\phi_T=t+t\tau+t\tau^2+\tau^3, 
	$$
	where $t$ denotes the image of $T$ under the canonical reduction map $A\to \F_\fp$. The minimal polynomial of $\pi$ is 
	$$
	P(x)=f(x)=x^3 +2T^2x^2 +(3T^4  +T^2+3T +1)x+4\fp
	$$
	From this we compute that 
	\begin{align*}\disc(A[\pi]) & =(T+4)^6(T^4 + 2T^3 + 4T^2 + 3T + 4),\\ 
	\disc(B) &=(T^4 + 2T^3 + 4T^2 + 3T + 4).
	\end{align*}
	Hence $\chi(B/A[\pi])=(T+4)^3$. We deduce that either $b_1=T+4$ and $b_2=(T+4)^2$, or 
$b_1=1$ and $b_2=(T+4)^n$ for some $0\leq n\leq 3$. The second step of our algorithm confirms that 
in fact $b_1=T+4$ and $b_2=(T+4)^2$. In particular, $\cE=B$. Moreover, the corresponding polynomials are 
$
f_1(x)=x+4$ and $f_2(x)=(x+4)^2$. 
An $A$-basis of $\cE$ is given by 
$$
e_1=1, \quad e_2=\frac{\pi+4}{T+4}, \quad e_3=e_2^2. 
$$
Finally, the element in $\F_\fp\{\tau\}$ corresponding to $e_2$ is 
$$
e_2 = \tau^3 + (2t^5 + 3t^4 + t + 1)\tau^2 + (4t^3 + 2t + 3)\tau + t^5 + 4t^4 + 4t^3 + 
4t^2 + 3.$$
\end{example}

%\begin{rem} The complexity of our algorithm is not very good. Rough estimates show that the running time is exponential in $\# k$ and $r$. Especially time consuming is going through all possible choices of the polynomials $f_i$. On the other hand, we expect that some parts of the algorithm could be optimized, leading to a better running time.  \end{rem}

\begin{rem}\label{remKP}
	In \cite{KP}, Kuhn and Pink gave a different algorithm for computing $\End_k(\phi)$. They work in 
	the most general setting where $A$ is a finitely generated normal integral domain of transcendence degree $1$ over a finite field $\F_q$, 
	$k$ is an arbitrary finitely generated field, 
	and no restrictions on $D$ are imposed. On the other hand, the emphasis 
	of \cite{KP} is on the existence of a deterministic algorithm that computes $\End_k(\phi)$ rather that its practicality, so some of the details of the 
	algorithm are left out. 
	
	In the case where $A=\F_q[T]$ and $k$ is finite, the approach of Kuhn and Pink is the following. 
	Consider $k\{\tau\}$ as a free module of finite rank over $R:=\F_q[\pi]$. (Note that $R$ is the center of $k\{\tau\}$). 
	Choose an $R$-basis of $k\{\tau\}$. For example, if 
	$n=[k:\F_q]$ and $\alpha_0, \dots, \alpha_{n-1}$ is an $\F_q$-basis of $k$, then 
	$\{\beta_{ij}:=\alpha_i\tau^j \mid 0\leq i,j\leq n-1\}$ 
	is a basis of $k\{\tau\}$ over $R$. Express $\phi_T\in k\{\tau\}$ in terms of this basis $\phi_T=\sum_{0\leq i,j\leq r-1}m_{ij}\beta_{ij}$. 
	Let $u=\sum_{0\leq i,j\leq r-1}x_{ij}\beta_{ij}$, where $x_{ij}$ are indeterminates. Now $u\in \End_k(\phi)$ if and only if 
	$u\phi_T=\phi_Tu$. This leads to a system of linear equations in $x_{ij}$'s (note that $\beta_{ij}$'s do not commute with 
	each other, so expanding both sides $u\phi_T=\phi_Tu$ in terms of the chosen basis leads to nontrivial linear equations 
	for $x_{ij}$). Choosing a basis for the space of solutions of the resulting system of linear equations gives a basis $e_1, \dots, e_s\in k\{\tau\}$ 
	of $\End_k(\phi)$ as an $R$-module. Next, one computes the action of $\phi_T$ on on this basis, which 
	gives an explicit matrix for the action of $\phi_T$ as an $R$-linear transformation of $\End_k(\phi)$. 
	It is then claimed in \cite{KP} (proof of Proposition 5.14) that this calculation yields a basis of $\End_k(\phi)$ as an $A$-module, 
	although the details of this deduction are not explained. We have not pursued this line of calculations, so we are unable to say 
	how complicated it is in practice, and whether it suffices for deducing the finer number-theoretic properties of 
	$\End_k(\phi)$, such as its discriminant over $A$ or the Frobenius index. 
\end{rem}

\section{Matrix of the Frobenius automorphism}\label{sFrob}

Let the notation and assumptions be as in Section \ref{sEndom}. In particular, $\phi$ 
is a Drinfeld module of rank $r$ over a finite extension $k$ of $\F_\fp$, and $\phi$ satisfies \eqref{eq-assumption}. 
Moreover, $\cE=\End_k(\phi)$, $A[\pi]\subset \cE$ 
is the suborder generated by the Frobenius endomorphism of $\phi$, and $B$ is the integral closure of $A$ in $F(\pi)$.

Let 
$$
P(x)=x^r+c_{r-1}x^{r-1}+\cdots+c_1x+c_0
$$ 
be the minimal polynomial of $\pi$ over $A$, and 
$$
f_{i}(x)=x^i+\sum_{j=0}^{i-1}a_{ij}x^j, \qquad  1\leq i\leq r-1
$$
be the polynomial from Theorem \ref{thm2}. Multiplication by $\pi$ induces an $A$-linear 
transformation of $\cE$. The matrix of this transformation with respect to the basis in Theorem \ref{thm2} has the form 
\begin{equation}\label{eqF_k}
\cF_k:=\begin{bmatrix} 
-a_{10} & \ast & \cdots & \ast & \ast \\ 
b_1 & a_{10}-a_{21} & \cdots & \ast & \ast\\
0 & \frac{b_2}{b_1} & a_{21}-a_{32} & \ast & \ast  \\ 
\vdots &  & \ddots  & &  \\ 
0 & 0 & \cdots & \frac{b_{r-1}}{b_{r-2}} &a_{r-1, r-2}-c_{r-1}
\end{bmatrix}. 
\end{equation}
(If $\cE=A[\pi]$, then $b_1=\cdots=b_{r-1}=1$, all $a_{ij}=0$, and $\cF_k$ is simply the companion matrix of $P(x)$.)

\begin{comment}
	$$
	\begin{bmatrix} 
	0 & 0 & \cdots & 0 & -c_0 \\ 
	1 & 0 & \cdots & 0& -c_1\\
	0 & 1 & \cdots & 0& -c_2 \\ 
	\vdots &  & \ddots  & &  \\ 
	0 & 0 & \cdots & 1 &-c_{n-1} \\ 
	\end{bmatrix}. 
	$$
\end{comment}

\begin{example}
	The entries of $\cF_k$ marked by $\ast$ are complicated expressions in coefficients of $f_i$ and $P$. 
	
	For $r=2$, the full matrix is 
	$
	\begin{bmatrix} 
	-a_{10} & \frac{-a_{10}(a_{10}-c_1)-c_0}{b_1} \\ 
	b_1 & a_{10}-c_1
	\end{bmatrix}$. 
	When $r=2$ and $q$ is odd, it is easy to show that $a_{10}=c_1/2$. It then follows from \eqref{eq-discchi} that $\cF_k$ is 
	$\begin{bmatrix} 
	-\frac{c_1}{2} & \frac{b_1\cdot \disc(\cE)}{4} \\ 
	b_1 & -\frac{c_1}{2}
	\end{bmatrix}$.
	%The conjugate of this matrix by $\begin{bmatrix} 1/2 & 0 \\ 0 & 1\end{bmatrix}$ is the matrix from Theorem 1 in \cite{CP}. 
	
	For $r=3$, the full matrix is 
\begin{equation}\label{eqFrobr=3}
\begin{bmatrix} 
-a_{10} & \frac{a_{10}(a_{21}-a_{10})-a_{20}}{b_1} & 
\frac{a_{10}a_{21}(a_{21}-c_2)-a_{10}(a_{20}-c_1)-a_{20}(a_{21}-c_2)-c_0}{b_2} \\ 
b_1 & a_{10}-a_{21} &  \frac{(a_{20}-c_1)-a_{21}(a_{21}-c_2)}{b_2/b_1}\\
0 & \frac{b_2}{b_1} & a_{21}-c_2  \\ 
\end{bmatrix}. 
\end{equation}
 As an explicit example, the matrix corresponding to Example \ref{example3.3} 
is 
\begin{equation}\label{eqFpEx3.3}
\begin{bmatrix} 
1 & 0 & T^4+T^2+2T+1\\ 
T+4 & 1 &  2T^3 +2T^2 +2T+4\\
0 & T+4 & 3(T^2+1) \\ 
\end{bmatrix}. 
\end{equation}
\end{example}

\begin{rem}
Even though fractions appear in $\cF_k$, all entries of this matrix are in $A$. This implies that there are non-obvious congruence  
	relations between the coefficients of $f_i$, $P$, and the Frobenius indices $b_j$. For example, from \eqref{eqFrobr=3} we get 
	$a_{10}(a_{21}-a_{10})\equiv a_{20}\Mod{b_1}$. Also note that by Proposition \ref{thm3}, $b_1$ divides all $b_{i}/b_{i-1}$, $2\leq i\leq r-1$, appearing 
	below the main diagonal in $\cF_k$, so if $n\mid b_1$, then $\cF_k$ is upper-triangular modulo $n$. In fact, it follows 
	from Theorem 3.1 in \cite{GP} that if $n\mid b_1$ then $\cF_k$ modulo $n$ is a scalar matrix. 
	%(assuming the condition in Theorem \ref{thmFrobMatrix} is satisfied).  
\end{rem}

Let $\fl\lhd A$ be a prime different from $\fp$. 
The arithmetic Frobenius automorphism $\Frob_k\in \Gal(\bar{k}/k)$ naturally acts on 
$T_\fl(\phi)$. Let $\mathrm{ch}_k(x)\in A_\fl[x]$ denote the characteristic polynomial of $\Frob_k\in \Aut_{A_\fl}(T_\fl(\phi))\cong \GL_r(A_\fl)$. 
The conjugacy class of $\Frob_k$ in $\Aut_{F_\fl}(T_\fl(\phi)\otimes F_\fl)\cong \GL_r(F_\fl)$ is uniquely 
determined by $\mathrm{ch}_k(x)$ because $T_\fl(\phi)\otimes F_\fl$ is a semi-simple $F_\fl[\Frob_k]$-module; cf. \cite{Taguchi}. 
On the other hand, $\mathrm{ch}_k(x)$ alone is not sufficient for determining the 
conjugacy class of $\Frob_k$ in $\Aut_{A_\fl}(T_\fl(\phi))$. 

\begin{thm}\label{thmTateFrob}
	Assume $T_\fl(\phi)$, under the natural action of 
	$\cE_\fl:=\cE\otimes_A A_\fl$, is a free module of rank $1$. 
	Then the matrix $\cF_k$ describes the action of $\Frob_k$ on $T_\fl(\phi)$, with respect to a suitable $A_\fl$-basis.  
\end{thm}
\begin{proof} 
	The action of $\Frob_k$ on $T_\fl(\phi)$ 
	agrees with the action induced by $\pi\in \cE$. 
	If the assumption of the theorem holds, then there is an isomorphism $T_\fl(\phi)\cong \cE_\fl$ compatible 
	with the actions of $\pi$ on both sides. For the choice of a basis of $\cE$ from Theorem \ref{thm2}, $\pi$ acts on $\cE$ by the 
	matrix $\cF_k$. 
\end{proof}

\begin{rem}
Note that $\cE_\fl\otimes F_\fl$  is a semi-simple $F_\fl$-algebra which 
acts faithfully on $T_\fl(\phi)\otimes F_\fl$, so $T_\fl(\phi)\otimes F_\fl$ is free of rank $1$ over $\cE_\fl\otimes F_\fl$. On the other hand, as we will see later 
in this section, $T_\fl(\phi)$  is not always free over $\cE_\fl$. Also note 
that in our case the characteristic polynomial $\mathrm{ch}_k(x)$  is the minimal polynomial $P(x)$ 
of $\pi$ over $A$. 
\end{rem}

Let $\Phi: A\to F\{\tau\}$ be a Drinfeld module of rank $r$ over $F$. Let $\fp$ be a prime of good reduction of $\Phi$. 
Denote by $\phi$ the reduction of $\Phi$ modulo $\fp$. Let $\cE=\End_{\F_\fp}({\phi})$ and $A[\pi]\subset \cE$ 
be its suborder generated by the Frobenius endomorphism $\pi= \tau^{\deg(\fp)}$ of ${\phi}$. 
%Let $$P(x)=x^r+c_{r-1}x^{r-1}+\cdots+c_1x+c_0$$ be the minimal polynomial of $\pi$ over $A$. 
Note that since we are working over the field $\F_\fp$, the assumption \eqref{eq-assumption} is satisfied for $\phi$;  
cf. \cite[Prop. 2.1]{GP}. Denote by $\cF(\fp)$ the matrix \eqref{eqF_k} for $\phi$ over $\F_\fp$. 

\begin{thm}\label{thmFrobMatrix}
	Let $n\in A$ be a nonzero element not divisible by $\fp$. The Galois extension $F(\Phi[n])/F$ is unramified at $\fp$. 
	Suppose for every prime $\fl\lhd A$ dividing $n$ the Tate module $T_\fl({\phi})$ is a free $\cE_\fl$-module of rank $1$. Then 
	the integral matrix $\cF(\fp)$, when reduced modulo $n$, represents the class of the Frobenius at $\fp$ 
	in $\Gal(F(\Phi[n])/F)\subseteq \GL_r(A/nA)$. 
\end{thm} 
\begin{proof}
	The fact that $\fp$ is unramified in $F(\Phi[n])/F$ is well-known, since $\fp$ is a prime of good reduction for $\Phi$ and does not divide $n$; 
	cf. \cite{Takahashi}. In fact, by \cite{Takahashi}, the Tate module $T_\fl(\Phi)$ is unramified at $\fp$, i.e., for any place $\bar{\fp}$ in $F^\sep$ 
	extending $\fp$, the inertia group of $\bar{\fp}$ acts trivially on $T_\fl(\Phi)$. There is a canonical isomorphism 
	$T_\fl(\Phi)\cong T_\fl({\phi})$ which is compatible with the action of 
	a Frobenius element in the decomposition group of $\bar{\fp}$ on $T_\fl(\Phi)$ and the action of the arithmetic 
	Frobenius automorphism $\Frob_\fp\in \Gal(\overline{\F}_\fp/\F_\fp)$ on $T_\fl({\phi})$; cf. \cite[p. 479]{Takahashi}. On the other hand, the action of $\Frob_\fp$ on $T_\fl({\phi})$ agrees 
	with the action induced by $\pi\in \cE$. 
	
	If the assumption of the theorem holds, then there is an isomorphism ${\phi}[n]\cong \cE/n\cE$ compatible 
	with the actions of $\pi$ on both sides. For the choice of a basis of $\cE$ from Theorem \ref{thm2}, $\pi$ acts on $\cE/n\cE\cong (A/nA)^r$ by the 
	matrix $\cF(\fp)$ reduced modulo $n$. Combining this with the isomorphism $\Phi[n]\cong {\phi}[n]$ compatible with 
	the action of Frobenius automorphism on both sides, we see that $\cF(\fp)\Mod{n}$ indeed represents the class of the Frobenius at $\fp$ 
	in $\Gal(F(\Phi[n])/F)\subseteq \GL_r(A/nA)$.  
\end{proof}

Theorems \ref{thmTateFrob} and \ref{thmFrobMatrix} are analogues of a result  of Duke and T\'oth \cite{DT} for elliptic curves (see also 
Theorems 2 and 3 in \cite{Centeleghe}). 

Theorem \ref{thmFrobMatrix} essentially says that the matrix $\cF(\fp)\in \Mat_r(A)$ is a ``universal'' matrix 
of the Frobenius automorphism at $\fp$ in the division fields of $\Phi$, in the sense that to get a matrix in the conjugacy class of 
the Frobenius in the Galois groups of different division fields $F(\Phi[n])$  we just need to reduce $\cF(\fp)$ 
modulo the correspond $n$. But there is a technical assumption in the theorem about the freeness 
of the Tate modules of $\phi$ as modules over the endomorphism ring of ${\phi}$. For the rest of this section 
we examine this assumption more carefully and show that it is a mild assumption, although quite subtle.  
Our considerations are motivated by \cite[$\S$4]{ST} and \cite{Centeleghe}. 

\begin{defn}\label{defGor} Let $\fl\lhd A$ be a prime. Let $R$ be a finite flat local $A_\fl$-algebra with maximal ideal $\fM$. 
	Put $\bar{R}=R/\fl R$. It is an artinian local ring. We denote its maximal ideal by $\overline{\fM}$. 
	The following statements are equivalent (see \cite[Prop. 1.4]{Tilouine}, \cite{Bass}, \cite[$\S$18]{Matsumura}): 
	\begin{enumerate}
		\item $\Hom_{A_\fl}(R, A_\fl)$ is free of rank $1$ over $R$. 
		\item $\Hom_{\F_\fl}(\bar{R}, \F_\fl)$ is free of rank $1$ over $\bar{R}$. 
		\item $\bar{R}[\overline{\fM}]=\{a\in \bar{R}\mid ma=0 \text{ for all }m\in \overline{\fM}\}$ 
		is $1$-dimensional over $\overline{R}/\overline{\fM}$. 
	\end{enumerate}
	We say that $R$ is \textit{Gorenstein} if it satisfies these conditions. We say that a finite flat (not necessarily local) $A_\fl$-algebra $R$ is 
	\textit{Gorenstein} if its localization at every maximal ideal is a Gorenstein local ring. 
\end{defn}

\begin{thm}
	Let $\fl\lhd A$ be a prime different from $\fp$. If $\cE_\fl$ is a Gorenstein ring, then $T_\fl(\phi)$ is a free $\cE_\fl$-module 
	of rank $1$. 
\end{thm}
\begin{proof} %For the definition and general properties of Gorenstein rings we refer to \cite{Bass} and \cite[$\S$18]{Matsumura}. 
	
	The ring $\cE_\fl$ is a finite flat $A_\fl$-algebra. 
	The module $T_\fl(\phi)$ is a torsion-free $\cE_\fl$-module.  
	Suppose $\cE_\fl$ is Gorenstein. Then by Theorem 6.2 and Proposition 7.2 in \cite{Bass}, either $T_\fl(\phi)$ 
	is a projective $\cE_\fl$-module or $T_\fl(\phi)$ is an $\cE_\fl'$-module for some $\cE_\fl\subsetneq \cE_\fl'\subset B\otimes_A A_\fl$. 
	Suppose the latter is the case. Let $G:=\Gal(\overline{\F}_\fp/\F_\fp)$. By \cite[Thm. 2]{JKYu}, we have an isomorphism 
	\begin{equation}\label{eqTateIsom}
	\cE_\fl \xrightarrow{\sim} \End_{A_\fl[G]}(T_\fl(\phi)). 
	\end{equation}
	Since $\cE_\fl'\otimes F_\fl=\cE_\fl\otimes F_\fl$ and $\cE_\fl\otimes F_\fl \xrightarrow{\sim} \End_{F_\fl[G]}(T_\fl(\phi)\otimes F_\fl)$, 
	the action of $\cE_\fl'$ on $T_\fl(\phi)$ has to commute with the action of $G$. Hence $\cE_\fl'\subseteq \End_{A_\fl[G]}(T_\fl(\phi))=\cE_\fl$. 
	This contradicts our earlier assumption. We conclude that $T_\fl(\phi)$ 
	is a projective $\cE_\fl$-module. Since $\cE_\fl$ is a semilocal ring, a projective module over $\cE_\fl$ is free 
	by \cite[Thm. 2.5 ]{Matsumura}. In particular, $T_\fl(\phi)$ is free. Finally, since the ranks of $T_\fl(\phi)$ and $\cE_\fl$ 
	over $A_\fl$ are the same, $T_\fl(\phi)$ is a free $\cE_\fl$-module of rank $1$. 
\end{proof}

\begin{prop}\label{propGor} Suppose one of the following conditions holds:
	\begin{enumerate}
		\item $\cE_\fl=A_\fl[\pi]$. 
		\item $r=2$.
		\item $\cE_\fl=B\otimes_A A_\fl$. 
	\end{enumerate}
	Then $\cE_\fl$ is Gorenstein. In particular, if $\fl$ does not divide $\chi(\cE/A[\pi])$ or $\chi(B/\cE)$, then $\cE_\fl$ is Gorenstein.
\end{prop}
\begin{proof}
	(1) If $\cE_\fl$ is generated over $A_\fl$ by one element, then $\cE_\fl$ is Gorenstein; cf. \cite[p. 329]{Tilouine}. 
	(2) If $r=2$, then $\cE_\fl$ is obviously generated by one element (any of the elements which is not in $A_\fl\subset \cE_\fl$). 
	(3) $B\otimes_A A_\fl$ is a product of discrete valuation rings, and such rings are Gorenstein (see again \cite[p. 329]{Tilouine}). 
\end{proof}

%, so $T_\fl({\phi})$ is a free $\cE_\fl$-module of rank $1$. 

\begin{example}\label{example4.6}
Let $q=5$ and $\Phi:A\to F\{\tau\}$ be given by $\Phi_T=T+T\tau+T\tau^2+\tau^3$. Every prime of $A$ is a prime of good reduction for $\Phi$. Take 
	$\fp=T^6+3T^5+T^2+3T+3$. Then ${\phi}$, the reduction of $\Phi$ modulo ${\fp}$, is the Drinfeld module from Example \ref{example3.3}. We showed that in this case 
	$\cE=B$, so the assumption of Theorem \ref{thmFrobMatrix} is satisfied for every prime $\fl\neq \fp$. Note that the 
	matrix \eqref{eqFpEx3.3}, which is $\cF(\fp)$ associated to ${\phi}$, is congruent to the identity matrix modulo $n$ if and only if $n=T+4$. 
	This means that $\fp$ splits completely in $F(\Phi[n])$ if and only if $n=T+4$. 
\end{example}

\begin{example}
	We give an example where $\cE_\fl$ is not Gorenstein by changing $\fp$ in Example \ref{example4.6}.  
	Let $\Phi$ be as in that example, but $\phi$ be the reduction of $\Phi$ modulo $\fp=T^6 + 4T^4 + 4T^2 + T + 1$. 
	The minimal polynomial of $\pi$ is 
	$$
	P(x)=x^3 +2T^2x^2+(3T^4 + 2T^3 + 2T^2 + 1)x + 4\fp,
	$$
	and 
	\begin{align*}\disc(A[\pi]) & =(T+ 4)^6(T^4 + 3T^3 + T^2 + 2),\\ 
	\disc(B) &=T^4 + 3T^3 + T^2 + 2.
	\end{align*}
	Our algorithm shows that 
	$$b_1=1,\quad b_2=T+4, \quad \chi(\cE/A[\pi])=T+4, \quad \chi(B/\cE)=(T+4)^2, 
	$$
	and  an $A$-basis of $\cE$ is given by 
	$$e_1=1,\quad  e_2=\pi+4, \quad e_3=\frac{(\pi+4)^2}{T+4}. 
	$$
	(Although we will not need this, an $A$-basis of $B$ is given by $e_1, e_2/(T+4), e_3/(T+4)$.)
	
	Let $\fl=T+4$. We claim that $\cE_\fl$ is not Gorenstein. By a routine calculation one obtains the relations 
	\begin{align*}
	e_2^2 &= (T+4)e_3\\ 
	e_3^2 &=(T+1)(T+3)(T+4)^2(T^2+2)e_1 + (T+4)(T^3+3T+2)e_2 + (T+4)T^2e_3 \\ 
	e_2e_3 &= (T+3)(T+4)^2(T^2+2)e_1 + 2(T+2)(T+4)^2e_2 + 3(T+1)(T+4)e_3. 
	\end{align*}
	From this it is easy to see that $\cE_\fl$ is local with maximal ideal $\fM=(\fl, e_2, e_3)$. To prove that $\cE_\fl$ 
	is not Gorenstein, we check (3) from Definition \ref{defGor}. 
	In our case, $\bar{\cE}_\fl=\F_\fl+\F_\fl \bar{e}_2+\F_\fl \bar{e}_3$, $\overline{\fM}=(\bar{e}_2, \bar{e}_3)$, and 
\begin{equation}\label{eqe2e3}
	\bar{e}_2^2=\bar{e}_3^2=\bar{e}_2 \bar{e}_3=0. 
\end{equation}
	Hence $\bar{\cE}_\fl[\overline{\fM}]=\overline{\fM}$ is two-dimensional over $\F_\fl$, so $\cE_\fl$ is not Gorenstein.  

	Of course, the fact that $R:=\cE_\fl$ is not Gorenstein does not necessarily imply that $M:=T_\fl(\phi)$ is not free over $R$. For that, one 
	needs an additional calculation.  By a standard argument involving Nakayama's lemma 
one shows that $M$ is a free $R$-module of rank $1$ if and only if $\bar{M}=M/\fl M={\phi}[\fl]$ is a free $\bar{R}=R/\fl R$-module of rank $1$. 
We need to compute the action of $\bar{R}$ on $\phi[\fl]$ as a $3$-dimensional vector space over $A/\fl\cong \F_5$. 

Now $\phi[\fl]$ is the set of roots of the polynomial $\phi_\fl(x)=x^{125}+tx^{25}+tx^5+(t+4)x\in \F_\fp[x]$, 
where $t$ is the image of $T$ in $\F_\fp$. This polynomial decomposes over $\F_\fp$ into a product of irreducible 
polynomials all of which have either degree $1$ or degree $5$. One of the irreducible factors of $\phi_\fl(x)$ of degree $5$ is 
$g(x)=x^5 + (3t^3 + 2t^2 + 2t)x + t^5 + 3t^4 + 3t^2 + 2t$. Let $\alpha$ be a root of $g(x)$. Then the splitting field of $\phi_\fl(x)$ 
is $\F_\fp(\alpha)$. The following is an $\F_5$-basis of $\phi[\fl]$ 
in $\F_\fp(\alpha)$: 
$$
v_1 = t^5 + 2t^3 + t^2 + 3,\quad 
v_2 = \alpha, \quad 
v_3 = \alpha + t^5 + 3t^3 + 4t. 
$$
(We simply chose three, more-or-less random, roots $v_1, v_2, v_3$ of $\phi_\fl(x)$ and verified that they are linearly independent over $\F_5$.) 
To compute the action of $e_2$ and $e_3$ on $\phi[\fl]$ we use their explicit expressions in $\F_\fp\{\tau\}$ provided by our algorithm
\begin{align*}
e_2 & = \tau^6+4\\
e_3 & = \tau^9 +(3t^5 +2t^3 +2t^2 +1)\tau^8 +(t^5 +t^4 +4t^3 +4t^2 +t+3)\tau^7 \\ 
&+(3t^5 +t^4 +2t^2 +3t+1)\tau^6 + (2t^5 +3t^4 +3t^3 +t^2 +3)\tau^5 \\ 
& +(3t^4 +2t^3 +2t+4)\tau^4 +(4t^5 +t^2 +3t+2)\tau^3 +(2t^4 +t^3 +t^2 +4t+4)\tau^2 \\ 
& + (3t^5 +t^4 +3t^3 +4t^2 +t+1)\tau +4t^5 +4t^4 +t. 
\end{align*}
With respect to the basis $\{v_1, v_2, v_3\}$  (as column vectors), $e_2$ and $e_3$ correspond to the following matrices: 
$$
%\bar{e}_1=\begin{bmatrix} 1 & 0 & 0 \\ 0 & 1 & 0 \\ 0 & 0 & 1\end{bmatrix}, \qquad 
\bar{e}_2=\begin{bmatrix} 0 & 0 & 0 \\ 0 & 3 & 3 \\ 0 & 2 & 2\end{bmatrix}, \qquad 
\bar{e}_3=\begin{bmatrix} 0 & 0 & 0 \\ 4 & 3 & 3 \\ 1 & 2 & 2\end{bmatrix}. 
$$

$\bar{M}$ is a free $\bar{R}$-module of rank $1$ 
if and only if there is a vector $v\in\F_5^3$ such that 
$\{v, \bar{e}_2 v, \bar{e}_3 v\}$ are linearly independent over $\F_5$. Since 
$\det \begin{bmatrix} 
\begin{bmatrix} x_1 \\ x_2 \\ x_3 \end{bmatrix} & \bar{e}_2\begin{bmatrix} x_1 \\ x_2 \\ x_3 \end{bmatrix}& \bar{e}_3 
\begin{bmatrix} x_1 \\ x_2 \\ x_3 \end{bmatrix} \end{bmatrix} = 0$,
%$$
%\det \begin{bmatrix} v & e_2v & e_3 v \end{bmatrix} =
%\det\begin{bmatrix} x_1 & 0 & 0 \\ x_2 & 3x_2+3x_3 & 4x_1+3x_2+3x_3\\ x_3 &  2x_2+2x_3  & x_1+2x_2+2x_3 \end{bmatrix}=0 
%$$
such a vector does not exists. 
Thus, in this example we encounter the strange phenomenon where 
$T_\fl(\phi)$ is not free over $\cE_\fl$. Note also that one can consider $\phi$ as a Drinfeld $\cE$-module 
of rank $1$ in the sense of Hayes \cite[p. 180]{Hayes}, and our calculation shows that, unlike the usual Drinfeld modules, 
${\phi}[\fl]$ is not isomorphic to $\cE/\fl\cE$ as an $\cE$-module. 

The matrix $\cF(\fp)$ for this example is 
$$
\cF(\fp)=
\begin{bmatrix} 
0	&	4		&	(T+4)^2(T^3 + 3T^2 + 2) \\
1	&	2   		&	2(T+2)(T+4)^2\\
0	&	T + 4	&	3(T+2)(T+3)\\ 
\end{bmatrix}
$$
Hence $\cF(\fp)\equiv \begin{bmatrix}
0	&	4		&	0 \\
1	&	2   		&	0 \\
0	&	0		&	1 \\ 
\end{bmatrix}\Mod{\fl}$. With respect to the basis $\{v_1, v_2, v_3\}$ of $\phi[\fl]$, the action of $\pi$ on $\phi[\fl]$ is 
given by 
$\begin{bmatrix}
1	&	0		&	0 \\
0	&	4   		&	3 \\
0	&	2		&	3 \\ 
\end{bmatrix}$, which is conjugate to $\cF(\fp)\ \mod\ {\fl}$ in $\GL_3(\F_5)$.  Thus, the conclusion of Theorem \ref{thmFrobMatrix} 
is still valid for $n=T+4$, even though its assumption fails. 
\end{example}

%-----------------------------------------------------

%\renewcommand{\bibliofont}{\normalsize}
%\bibliographystyle{amsplain} 
\bibliographystyle{acm}
\bibliography{Endom2HR.bib}

\begin{thebibliography}{10}

\bibitem{Angles}
{\sc Angles, B.}
\newblock On some subrings of {O}re polynomials connected with finite
  {D}rinfeld modules.
\newblock {\em J. Algebra 181}, 2 (1996), 507--522.

\bibitem{Bass}
{\sc Bass, H.}
\newblock On the ubiquity of {G}orenstein rings.
\newblock {\em Math. Z. 82\/} (1963), 8--28.

\bibitem{Centeleghe}
{\sc Centeleghe, T.~G.}
\newblock Integral {T}ate modules and splitting of primes in torsion fields of
  elliptic curves.
\newblock {\em Int. J. Number Theory 12}, 1 (2016), 237--248.

\bibitem{CP}
{\sc Cojocaru, A.~C., and Papikian, M.}
\newblock Drinfeld modules, {F}robenius endomorphisms, and {CM}-liftings.
\newblock {\em Int. Math. Res. Not. IMRN}, 17 (2015), 7787--7825.

\bibitem{CS2}
{\sc Cojocaru, A.~C., and Shulman, A.~M.}
\newblock The distribution of the first elementary divisor of the reductions of
  a generic {D}rinfeld module of arbitrary rank.
\newblock {\em Canad. J. Math. 67}, 6 (2015), 1326--1357.

\bibitem{Drinfeld}
{\sc Drinfeld, V.~G.}
\newblock Elliptic modules.
\newblock {\em Mat. Sb. (N.S.) 94(136)\/} (1974), 594--627, 656.

\bibitem{Drinfeld2}
{\sc Drinfeld, V.~G.}
\newblock Elliptic modules. {II}.
\newblock {\em Mat. Sb. (N.S.) 102(144)}, 2 (1977), 182--194, 325.

\bibitem{DT}
{\sc Duke, W., and T\'{o}th, A.}
\newblock The splitting of primes in division fields of elliptic curves.
\newblock {\em Experiment. Math. 11}, 4 (2002), 555--565 (2003).

\bibitem{DF}
{\sc Dummit, D.~S., and Foote, R.~M.}
\newblock {\em Abstract algebra}, third~ed.
\newblock John Wiley \& Sons, Inc., Hoboken, NJ, 2004.

\bibitem{GP}
{\sc Garai, S., and Papikian, M.}
\newblock Endomorphism rings of reductions of {Drinfeld} modules.
\newblock {\em J. Number Theory\/}.
\newblock to appear.

\bibitem{GekelerCR}
{\sc Gekeler, E.-U.}
\newblock Sur les classes d'id\'{e}aux des ordres de certains corps gauches.
\newblock {\em C. R. Acad. Sci. Paris S\'{e}r. I Math. 309}, 9 (1989),
  577--580.

\bibitem{GekelerFDM}
{\sc Gekeler, E.-U.}
\newblock On finite {D}rinfeld modules.
\newblock {\em J. Algebra 141}, 1 (1991), 187--203.

\bibitem{GekelerTAMS}
{\sc Gekeler, E.-U.}
\newblock Frobenius distributions of {D}rinfeld modules over finite fields.
\newblock {\em Trans. Amer. Math. Soc. 360}, 4 (2008), 1695--1721.

\bibitem{Hayes}
{\sc Hayes, D.~R.}
\newblock Explicit class field theory in global function fields.
\newblock In {\em Studies in algebra and number theory}, vol.~6 of {\em Adv. in
  Math. Suppl. Stud.} Academic Press, New York-London, 1979, pp.~173--217.

\bibitem{KP}
{\sc Kuhn, N., and Pink, R.}
\newblock Finding endomorphisms of {D}rinfeld modules.
\newblock {\em arXiv:1608.02788\/}.

\bibitem{KL}
{\sc Kuo, W., and Liu, Y.-R.}
\newblock Cyclicity of finite {D}rinfeld modules.
\newblock {\em J. Lond. Math. Soc. (2) 80}, 3 (2009), 567--584.

\bibitem{LaumonCDV}
{\sc Laumon, G.}
\newblock {\em Cohomology of {D}rinfeld modular varieties. {P}art {I}}, vol.~41
  of {\em Cambridge Studies in Advanced Mathematics}.
\newblock Cambridge University Press, Cambridge, 1996.
\newblock Geometry, counting of points and local harmonic analysis.

\bibitem{Marcus}
{\sc Marcus, D.~A.}
\newblock {\em Number fields}.
\newblock Universitext. Springer, Cham, 2018.
\newblock Second edition, With a foreword by Barry Mazur.

\bibitem{Matsumura}
{\sc Matsumura, H.}
\newblock {\em Commutative ring theory}, second~ed., vol.~8 of {\em Cambridge
  Studies in Advanced Mathematics}.
\newblock Cambridge University Press, Cambridge, 1989.
\newblock Translated from the Japanese by M. Reid.

\bibitem{Mazur}
{\sc Mazur, B.}
\newblock Modular curves and the {E}isenstein ideal.
\newblock {\em Inst. Hautes \'{E}tudes Sci. Publ. Math.}, 47 (1977), 33--186
  (1978).
\newblock With an appendix by Mazur and M. Rapoport.

\bibitem{SerreLF}
{\sc Serre, J.-P.}
\newblock {\em Local fields}, vol.~67 of {\em Graduate Texts in Mathematics}.
\newblock Springer-Verlag, New York-Berlin, 1979.
\newblock Translated from the French by Marvin Jay Greenberg.

\bibitem{ST}
{\sc Serre, J.-P., and Tate, J.}
\newblock Good reduction of abelian varieties.
\newblock {\em Ann. of Math. (2) 88\/} (1968), 492--517.

\bibitem{Taguchi}
{\sc Taguchi, Y.}
\newblock Semisimplicity of the {G}alois representations attached to {D}rinfeld
  modules over fields of ``finite characteristics''.
\newblock {\em Duke Math. J. 62}, 3 (1991), 593--599.

\bibitem{Takahashi}
{\sc Takahashi, T.}
\newblock Good reduction of elliptic modules.
\newblock {\em J. Math. Soc. Japan 34}, 3 (1982), 475--487.

\bibitem{Tilouine}
{\sc Tilouine, J.}
\newblock Hecke algebras and the {G}orenstein property.
\newblock In {\em Modular forms and {F}ermat's last theorem ({B}oston, {MA},
  1995)}. Springer, New York, 1997, pp.~327--342.

\bibitem{JKYu}
{\sc Yu, J.-K.}
\newblock Isogenies of {D}rinfeld modules over finite fields.
\newblock {\em J. Number Theory 54}, 1 (1995), 161--171.

\end{thebibliography}

\end{document}